\documentclass[11pt, twoside]{article}

\usepackage{amssymb}
\usepackage{amsmath}
\usepackage{amsthm}
\usepackage{color}
\usepackage{mathrsfs}
\usepackage{accents}
\usepackage{txfonts}

\usepackage{indentfirst}

\allowdisplaybreaks

\def\rr{{\mathbb R}}
\def\rn{{\mathbb{R}^n}}
\def\zz{{\mathbb Z}}

\def\nn{{\mathbb N}}

\def\cs{{\mathcal S}}

\def\fz{\infty }
\def\az{\alpha}

\def\gz{\gamma}
\def\lz{\lambda}

\def\lf{\left}
\def\r{\right}
\def\hs{\hspace{0.25cm}}
\def\ls{\lesssim}
\def\noz{\nonumber}

\def\supp{\mathop\mathrm{\,supp\,}}
\def\esup{\mathop\mathrm{\,ess\,sup\,}}
\def\B{\mathfrak{B}}

\def\vh{{H_A^{\vec{p}}(\rn)}}

\def\vah{{H_A^{\vec{p},r,s}(\rn)}}

\def\lv{{L^{\vec{p}}(\rn)}}

\newtheorem{theorem}{Theorem}[section]
\newtheorem{lemma}[theorem]{Lemma}

\theoremstyle{definition}
\newtheorem{remark}[theorem]{Remark}
\newtheorem{definition}[theorem]{Definition}
\renewcommand{\appendix}{\par
   \setcounter{section}{0}%
   \setcounter{subsection}{0}%
   \setcounter{subsubsection}{0}%
   \gdef\thesection{\@Alph\c@section}%
   \gdef\thesubsection{\@Alph\c@section.\@arabic\c@subsection}%
   \gdef\theHsection{\@Alph\c@section.}%
   \gdef\theHsubsection{\@Alph\c@section.\@arabic\c@subsection}%
   \csname appendixmore\endcsname
 }

\numberwithin{equation}{section}

\begin{document}

\arraycolsep=1pt

\title{\bf\Large Fourier Transform of Anisotropic Mixed-norm Hardy Spaces
with Applications to Hardy--Littlewood Inequalities
\footnotetext{\hspace{-0.35cm} 2020 {\it
Mathematics Subject Classification}. Primary 42B35;
Secondary 42B30, 42B10, 46E30.
\endgraf {\it Key words and phrases.}
dilation, mixed-norm Hardy space, Fourier transform,
Hardy--Littlewood inequality.
\endgraf This research was funded by the Natural Science Foundation of Jiangsu Province (Grant No. BK20200647), the National Natural Science Foundation of China (Grant No. 12001527) and the Project Funded by China Postdoctoral Science Foundation (Grant No. 2021M693422).}}
\author{Jun Liu\footnote{Corresponding author,
E-mail: \texttt{junliu@cumt.edu.cn}},\ \ Yaqian Lu and Mingdong Zhang}
\date{}
\maketitle

\vspace{-0.8cm}

\begin{center}
\begin{minipage}{13cm}
{\small {\bf Abstract}\quad
Let $\vec{p}\in(0,1]^n$ be a $n$-dimensional vector and $A$ a dilation.
Let $H_A^{\vec{p}}(\mathbb{R}^n)$ denote the anisotropic
mixed-norm Hardy space defined via the radial maximal function.
Using the known atomic characterization of
$H_{A}^{\vec{p}}(\mathbb{R}^n)$ and establishing a uniform estimate
for corresponding atoms, the authors prove that the Fourier transform
of $f\in H_A^{\vec{p}}(\mathbb{R}^n)$ coincides with a
continuous function $F$ on $\mathbb{R}^n$ in the sense of tempered distributions.
Moreover, the function $F$ can be controlled pointwisely by the product of
the Hardy space norm of $f$ and a step function with respect to the transpose
matrix of $A$. As applications, the authors obtain a higher order of convergence for the function
$F$ at the origin, and an analogue of Hardy--Littlewood inequalities in the present setting
of $H_A^{\vec{p}}(\mathbb{R}^n)$.}
\end{minipage}
\end{center}

\vspace{0.2cm}

\section{Introduction\label{s1}}

Let $\vec{p}:=(p_1,\ldots,p_n)\in(0,\fz)^n$ be a $n$-dimensional vector and $A$ a dilation (see Definition \ref{2d1} below).
The anisotropic mixed-norm Hardy space $\vh$ was introduced in \cite{hlyy20}.
The main purpose of this paper is to study the Fourier transform
on $\vh$ associated with $\vec{p}\in(0,1]^n$.
The question of the Fourier transform on classical Hardy spaces $H^p({\mathbb{R}^n})$ was put forward originally by Fefferman and Stein \cite{fs72}, which is an important topic in the real-variable theory of
$H^p({\mathbb{R}^n})$. Applying entire functions of exponential type,
Coifman \cite{coi74} first characterized the Fourier transform $\widehat{f}$ of
$f\in H^p({\mathbb R})$. The related conclusions in higher dimensions were studied in \cite{bw13,col82,gk01,tw80}. Particularly, the following estimate was given by
Taibleson and Weiss \cite{tw80}: for any given $p\in(0,1]$,
the Fourier transform of
$f\in H^p({\mathbb{R}^n})$ coincides with a continuous
function $F$ on $\rn$, which satisfies that there exists a positive constant $C_{(n,p)}$ such that,
for any $x\in\rn$,
\begin{align}\label{1e1}
\left|F(x)\right|\le C_{(n,p)}\|f\|_{H^p({\mathbb{R}^n})}|x|^{n(1/p-1)}.
\end{align}
Moreover, the estimate \eqref{1e1} illustrates the following inequality as a generalization of the well-konwn Hardy--Littlewood
inequality for Hardy spaces, that is, for any fixed $p\in(0,1]$, there
exists a positive constant $K$ such that, for each $f\in H^p({\mathbb{R}^n})$,
\begin{align}\label{1e2}
\left[\int_{{\mathbb{R}^n}}|x|^{n(p-2)}
\left|F(x)\right|^p\,dx\right]^{1/p}\le K\|f\|_{H^p({\mathbb{R}^n})},
\end{align}
where $F$ is as in \eqref{1e1}; see \cite[p.\,128]{ste93}.

On the other hand, the theory of classic Hardy spaces $H^p(\rn)$ has a wide range of applications
in many mathematical fields such as harmonic analysis and partial differential equations;
see, for instance, \cite{fs72,mul94,ste93,sw60}. Inspired by the notable work of Calder\'{o}n and Torchinsky \cite{ct75}
on parabolic Hardy spaces, there were various generalizations of classic Hardy spaces; see,
for instance, \cite{bow03,cgn17,gkp21,hlyy20,tri06,yyh13,whhy21}. In particular, Bownik \cite{bow03} introduced the anisotropic Hardy space $H^p_A({\mathbb{R}^n})$, where $p\in(0,\fz)$ and $A$ is a dilation, which is actually a generalization of both the isotropic Hardy space and the parabolic Hardy space. In addition, via the atomic characterization of $H^p_A({\mathbb{R}^n})$, Bownik and Wang \cite{bw13} extended both inequalities \eqref{1e1} and
\eqref{1e2} to the anisotropic Hardy space $H^p_A({\mathbb{R}^n})$.
Recently, the analogous results were proved in the new setting of Hardy spaces associated with ball quasi-Banach function spaces and
the anisotropic mixed-norm Hardy space ${H_{\vec{a}}^{\vec{p}}(\rn)}$,
where
$$\vec{a}:=(a_1,\ldots,a_n)\in [1,\infty)^n \quad
{\rm and} \quad\vec{p}:=(p_1,\ldots,p_n)\in (0,1]^n;$$
see, respectively, \cite{hcy21,hcy}. In addition, motivated by the previous work of \cite{cgn17,fs72,hlyy19},
Huang et al. \cite{hlyy20} introduced the anisotropic mixed-norm Hardy space $\vh$ with respect to
$\vec{p}\in(0,\fz)^n$ and a dilation $A$, and investigated its various real-variable characterizations.
For more information on mixed-norm function spaces, we refer the
reader to \cite{cs,cs20,cs21,cg20,cgn17-2,jms13,jms15,noss21}.

Inspired by the known results about the Fourier transform of the aforementioned Hardy-type spaces [namely, $H^p({\mathbb{R}^n})$, $H^p_A({\mathbb{R}^n})$ and ${H_{\vec{a}}^{\vec{p}}(\rn)}$], using the real-variable theory of the anisotropic mixed-normed Hardy space
$\vh$ from \cite{hlyy20}, in this paper, we extend the inequality \eqref{1e1} to the setting of anisotropic mixed-norm Hardy spaces $\vh$ and also  present some applications via our main result.

As a preliminary, in Section \ref{s2}, we present definitions of dilations,
mixed-norm Lebesgue spaces $\lv$ and anisotropic mixed-norm Hardy spaces
(see Definitions \ref{2d1}, \ref{2d4} and \ref{2d6} below).

Section \ref{s3} is aimed at proving the main result (see Theorem \ref{3t1} below), namely,
the Fourier transform $\widehat{f}$ of $f\in\vh$ coincides with a continuous
function $F$ in the sense of tempered distributions. To this end, applying Lemmas \ref{3l1} and \ref{3l3},
we first obtain a uniform pointwise estimate for atoms (see Lemma \ref{3l2} below).
Then, we use some real-variable characterizations from \cite{hlyy20}, especially atom decompositions,
to show Theorem \ref{3t1}.
Meanwhile, we also get a pointwise inequality
of the continuous function $F$, which indicates the necessity of vanishing moments of anisotropic mixed-norm atoms in some
sense [see Remark \ref{3r1}(ii) below].

As applications, in Section \ref{s4}, we present some consequences of Theorem \ref{3t1}.
First, the above function $F$ has a higher order convergence at the origin; see \eqref{4e1} below.
Moreover, we prove that the term
$$|F(\cdot)|\min\lf\{\lf[\rho_*(\cdot)\r]^{1-\frac1{p_-}-\frac1{p_+}},\,
\lf[\rho_*(\cdot)\r]^{1-\frac{2}{p_+}}\r\}$$
is $L^{p_+}$-integrable, and this integral can be uniformly controlled by a positive constant
multiple of the Hardy space norm of $f$; see \eqref{4e8} below.
The above result is actually a generalization of the Hardy--Littlewood inequality from classic Hardy spaces to the setting of anisotropic mixed-norm Hardy spaces.

Finally, we make some conventions on notation.
Let $\nn:=\{1,2,\ldots\}$, $\zz_+:=\{0\}\cup\nn$
and $\mathbf{0}$ be the \emph{origin} of $\rn$. For a given multi-index
$\az:=(\az_1,\ldots,\az_n)\in(\zz_+)^n=:\zz_+^n$,
let $|\az|:=\az_1+\cdots+\az_n$ and
$\partial^{\az}
:=(\frac{\partial}{\partial x_1})^{\az_1}\cdots(\frac{\partial}{\partial x_n})^{\az_n}.$
We use $C$ to denote a positive constant
which is independent of the main parameters, but may vary in different setting.
The \emph{symbol} $g\ls h$ means $g\le Ch$ and,
if $g\ls h\ls g$, then we write $g\sim h$. If $f\le Ch$ and $h=g$ or $h\le g$,
we then write $f\ls h\sim g$ or $f\ls h\ls g$, \emph{rather than} $f\ls h=g$
or $f\ls h\le g$. In addition,
for any set $E\subset\rn$, we denote its \emph{characteristic function} by $\mathbf{1}_E$,
the set $\rn\setminus E$ by $E^\complement$ and its \emph{n-dimensional Lebesgue measure} by $|E|$.
For any $s\in\mathbb{R}$, we use $\lfloor s\rfloor$
(resp., $\lceil s \rceil$) to denote the \emph{largest} (resp., \emph{least})
\emph{integer not} greater (resp., \emph{less}) \emph{than} $s$.

\section{Preliminaries \label{s2}}

In this section, we give the definitions of dilations, mixed-norm Lebesgue spaces and anisotropic mixed-norm Hardy spaces.
 The following definition is originally from \cite{bow03}.

\begin{definition}\label{2d1}

We call $A$ a \emph{dilation} if $A$ is a real $n\times n$ matrix $A$ and satisfies the following condition:
 $$\min_{\lz\in\sigma(A)}|\lz|>1,$$
where $\sigma(A)$ denotes the \emph{set of all eigenvalues of $A$}. We denote the eigenvalues of $A$ by $\lambda_1,\ldots,\lambda_n$, which satisfies $1<|\lambda_1|\leq\cdots\leq|\lambda_n|$. Here and thereafter, let $\lambda_-$ and $\lambda_+$ be two numbers such that
$1<\lambda_-<|\lambda_1|\leq\cdots\leq|\lambda_n|<\lambda_+$.
\end{definition}

By \cite[p.\,5, Lemma 2.2]{bow03}, for a given dilation $A$, there exists an open set in $\rn$ which is called an \emph{ellipsoid},
denoted by $\Delta$, and has the following property: $|\Delta|=1$, and we can find a constant $r\in(1,\infty)$ such that $\Delta\subset r\Delta\subset A\Delta$. For any given $i\in\zz$, we denote $A^i\Delta$ by $B_i$. It is easy to check that $\{B_i\}_{i\in\zz}$ is a family of open sets around the origin, $B_i\subset rB_i\subset B_{i+1}$ and $|B_i|=b^i$ with $b:=|\det A|$.
For any given dilation $A$, the notation $\mathfrak{B}$ is the set of all \emph{dilated balls}, namely,
\begin{align}\label{2e1}
\mathfrak{B}:=\lf\{x+B_i:\ x\in\rn,\ i\in\zz\r\}.
\end{align}

The next two definitions were introduced by Bownik \cite{bow03}.

\begin{definition}\label{2d2}
A measurable mapping $\rho:\ \rn \to [0,\infty)$
is called a \emph{homogeneous quasi-norm}, with respect to a dilation $A$, if
\begin{enumerate}
\item[\rm{(i)}] $\rho(x) \ge0$, and $\rho(x)=0 \Rightarrow x=0$;

\item[\rm{(ii)}] for any $x\in\rn$, $\rho(Ax)=b\rho(x)$;

\item[\rm{(iii)}] for any $x$, $y\in\rn$, $\rho(x+y)\le c[\rho(x)+\rho(y)]$, where $c$ is a positive constant independent of $x$ and $y$.
\end{enumerate}
\end{definition}

It is easy to verify that the following \emph{step homogeneous quasi-norm} is a homogeneous quasi-norm.

\begin{definition}\label{2d3}
A \emph{step homogeneous quasi-norm} $\rho$ with respect to a dilation $A$, is defined by setting, for each $x\in\rn$,
\begin{equation*}
\rho(x):=\left\{
\begin{array}{cl}
b^i&\hspace{0.6cm} {\rm when}\hspace{0.5cm} x\in
B_{i+1}\backslash B_i\\
0&\hspace{0.6cm} {\rm when}\hspace{0.5cm} x=\mathbf{0}
\end{array}\r..
\end{equation*}
\end{definition}

In \cite[p.\,5, Lemma 2.4]{bow03}, it was proved that any two homogeneous
quasi-norms associated with a fixed dilation $A$ are equivalent. For convenience,
in what follows, we always use the step homogeneous quasi-norm.

A $C^\fz$ complex-valued function $\phi$ on $\rn$ is called a \emph{Schwartz function} if, for
every pair of $k\in\zz_+$ and multi-index $\gz\in\zz_+^n$, the following inequality
$$\|\phi\|_{\gz,k}:=
\sup_{x\in\rn}[\rho(x)]^k\lf|\partial^\gz\phi(x)\r|<\infty$$
holds true.
The set of all Schwartz functions on $\rn$ is denoted by $\cs(\rn)$. Indeed, $\{\|\cdot\|_{\gz,k}\}_{\gz\in\zz_+^n,\,k\in\zz_+}$ is a family of semi-norms, which induces a topology and makes $\cs(\rn)$ to be a topological vector space. We denote the dual space of $\cs(\rn)$
by $\cs'(\rn)$, equipped with the weak-$\ast$ topology.

For a $n$-dimensional vector $\vec{p}:=(p_1,\ldots,p_n)\in{(0,\fz]}^n$, let
\begin{align}\label{2e3}
p_-:=\mathop\mathrm{min}_{i\in\{1,\ldots,n\}}\{p_i\},\hspace{0.35cm}
p_+:=\mathop\mathrm{max}_{i\in\{1,\ldots,n\}}\{p_i\},\hspace{0.35cm}
{\rm and}\hspace{0.35cm}\underline{p}:=\min\{p_-,1\}.
\end{align}

\begin{definition}\label{2d4}
Let $\vec{p}:=(p_1,\ldots,p_n)\in{(0,\fz]}^n$. The \emph{mixed-norm Lebesgue space} ${L}^{\vec{p}}(\rn)$  is defined to be the set of all measurable functions $f$ such that
$${\Arrowvert f\Arrowvert}_{{L}^{\vec{p}}(\rn)}:=\lf( \int_\rr \cdots\lf(\int_\rr|f(x_1,\ldots,x_n)|^{p_1} d x_1\r)^{\frac{p_2}{p_1}}\cdots d x_n \r)^{\frac{1}{p_n}} < \fz$$
with the usual modifications made when $p_i=\fz$ for some $i\in\{1,\ldots,n\}$.
\end{definition}

\begin{definition}\label{2d5}
Let $\varphi\in\cs(\rn)$ satisfy $\int_\rn\varphi(x)\,dx\neq0$. The
\emph{radial maximal function} $M_\varphi(f)$ of $f\in\cs'(\rn)$, with respect to $\varphi$, is defined by
\begin{align*}
M_\varphi(f)(x):= \sup_{k\in\zz}|f\ast\varphi_k(x)|, \quad \forall\, x\in\rn,
\end{align*}
here and thereafter, for any $\varphi\in\cs(\rn)$ and $k\in\zz$, $\varphi_k(\cdot):=b^{k}\varphi(A^{k}\cdot)$.
\end{definition}

\begin{definition}\label{2d6}
Let $\vec{p}\in{(0,\fz)}^n$ and $\varphi$ be as in Definition \ref{2d5}.
The \emph{anisotropic mixed-norm Hardy space} $\vh$ is defined by setting
\begin{equation*}
\vh:=\lf\{f\in\cs'(\rn):\ M_\varphi(f)\in\lv\r\}.
\end{equation*}
Moreover, for any $f\in\vh$, let
$\|f\|_{\vh}:=\| M_\varphi(f)\|_{\lv}$.
\end{definition}

\section{Fourier transforms of $\vh$\label{s3}}

In this section, we study the Fourier transform $\widehat{f}$ of $f\in\vh$. We first present
the notion of Fourier transforms.

For a given schwartz function $\varphi\in\cs(\rn)$, we define its \emph{Fourier transform} as follows:
$$\mathscr{F}\varphi(x)=\widehat\varphi(x)
:=\int_{\rn}\varphi(t)e^{-2\pi\imath t\cdot x}\,dt, \quad \forall\,x\in\rn,$$
where $\imath:=\sqrt{-1}$ and $t\cdot x := \sum_{k=1}^n t_k x_k$
for any $t:=(t_1,\ldots,t_n)$, $x:=(x_1,\ldots,x_n)\in\rn$.
Futhermore, we can also define the \emph{Fourier transform} of $f\in\cs'(\rn)$, also denoted by $\mathscr{F}f$ or $\widehat{f}$,
that is, for each $\varphi\in\cs(\rn)$,
$$\langle\mathscr{F}f,\varphi\rangle=\lf\langle\widehat f,\varphi\r\rangle
:=\lf\langle f,\widehat{\varphi}\r\rangle.$$

We now give the main result of this paper.
\begin{theorem}\label{3t1}
Let $\vec{p}\in{(0,1]}^n$. Then, for any
$f\in\vh$, there exists a continuous function $F$ on $\rn$ such that
$$\widehat{f}=F\quad in\quad \cs'(\rn),$$
and there exists a positive constant $C$, depending only on $A$ and $\vec{p}$, such that, for any $x\in\rn$,
\begin{align}\label{3e1}
|F(x)|\le C\|f\|_{\vh}\max\lf\{[\rho_\ast(x)]^{\frac1{p_-}-1},
[\rho_\ast(x)]^{\frac1{p_+}-1}\r\},
\end{align}
here and thereafter, $\rho_*$ is as in Definition \ref{2d2} with $A$ replaced by its
transposed matrix $A^*$.
\end{theorem}

Recall that, for a given measurable set $E\subset\rn$, the \emph{Lebesgue space}
$L^p(E)$, $0<p<\fz$, is the set of all the measurable functions satisfying that
$$\|f\|_{L^p(E)}:=\lf[\int_E|f(x)|^p\,dx\r]^{1/p}<\fz,$$
and $L^{\fz}(E)$ is the set of all the measurable functions satisfying that
$$\|f\|_{L^\fz(E)}:=\esup_{x\in E}|f(x)|<\fz.$$

The dilation operator $D_A$ is defined by setting, for any measurable function $f$ on $\rn$,
$$D_A(f)(\cdot):=f(A\cdot).$$
Then, for any $f\in L^1(\rn)$, $k\in\zz$ and $x\in\rn$, the following identity
$$\widehat{f}(x)=b^k\lf(D_{A^*}^k\mathscr{F}D_{A}^kf\r)(x)$$
can be easily verified.

Next, we present some notions used in the real-variable characterizations of anisotropic mixed-norm Hardy spaces; see \cite{hlyy20}.

\begin{definition}\label{3d3}
\begin{enumerate}
Let $\vec{p}\in{(0,\fz)}^n$, $r\in(1,\fz]$ and
\begin{align}\label{3e2}
s\in\lf[\lf\lfloor\lf(\dfrac1{p_-}-1\right)
\dfrac{\ln b}{\ln\lz_-}\right\rfloor,\fz\right)\cap\zz_+,
\end{align}
where $p_-$ is as in \eqref{2e3}.
\item[(\uppercase\expandafter{\romannumeral1})]
A measurable function $a$ on $\rn$ is called an \emph{anisotropic $(\vec{p},r,s)$-atom}
(simply, a $(\vec{p},r,s)$-\emph{atom}) if
\begin{enumerate}
\item[{\rm (i)}] $\supp a \subset B$, where
$B\in\B$ with $\B$ as in \eqref{2d1};

\item[{\rm (ii)}] $\|a\|_{L^r(\rn)}\le \frac{|B|^{1/r}}{\|\mathbf{1}_B\|_{\lv}}$;

\item[{\rm (iii)}] $\int_{\rn}a(x)x^\gz\,dx=0$ for any $\gz\in\zz_+^n$
with $|\gz|\le s$.
\end{enumerate}
\item[(\uppercase\expandafter{\romannumeral2})]
The \emph{anisotropic mixed-norm atomic Hardy space} $\vah$ is defined to be the
set of all $f\in\cs'(\rn)$ satisfying that there exist a sequence
$\{\lz_i\}_{i\in\nn}\subset\mathbb{C}$ and a sequence of $(\vec{p},r,s)$-atoms
$\{a_i\}_{i\in\nn}$, supported respectively in
$\{B^{(i)}\}_{i\in\nn}\subset\B$ such that
\begin{align*}
f=\sum_{i\in\nn}\lz_ia_i
\quad\mathrm{in}\quad\cs'(\rn).
\end{align*}
Furthermore, for any $f\in\vah$, let
\begin{align*}
\|f\|_{\vah}:=
{\inf}\lf\|\lf\{\sum_{i\in\nn}
\lf[\frac{|\lz_i|\mathbf{1}_{B^{(i)}}}{\|\mathbf{1}_{B^{(i)}}\|_{\lv}}\r]^
{\underline{p}}\r\}^{1/\underline{p}}\r\|_{\lv},
\end{align*}
where the infimum is taken over all the decompositions of $f$ as above.
\end{enumerate}
\end{definition}

The next Lemma \ref{3l1} will be used to prove Lemma \ref{3l2} below.

\begin{lemma}\label{3l1}
Let $\vec{p}$, $r$ and $s$ be as in Definition \ref{3d3}.
Assume that $a$ is a $(\vec{p},r,s)$-atom supported in $x_0+B_{i_0}$
with some $x_0\in\rn$ and $i_0\in\zz$. Then there exists a positive
constant $C$, depending only on $A$ and $s$, such that,
for any $\alpha\in\zz_+^n$ with $|\az|\leq s$ and $x\in\rn$,
\begin{align*}
\lf|\partial^\az\lf(\mathscr{F}D_{A}^{i_0}a\r)(x)\r|
\leq Cb^{-i_0/r}\|a\|_{L^r(\rn)}\min\lf\{1,|x|^{s-|\az|+1}\r\}{}.
\end{align*}
\end{lemma}

\begin{proof}
Without loss of generality, we may take $x_0=\mathbf{0}$. This implies that $\supp(D_{A}^{i_0}a)\subset B_0$.
Then, for any $\alpha\in\zz_+^n$ with $|\az|\leq s$ and $x\in\rn$, we easily obtain
$$\partial^\az\lf(\mathscr{F}D_{A}^{i_0}a\r)(x)=\int_{B_0}
(-2\pi\imath\xi)^\az\lf(D_{A}^{i_0}a\r)(\xi)e^{-2\pi\imath \xi\cdot x}\,d\xi.$$
Let $P$ be the degree $s-|\az|$ Taylor expansion polynomial of the function
$e^{-2\pi\imath \xi\cdot x}$ centered at the origin $\mathbf{0}$ respect to $\xi$.
Applying Definition \ref{3d3}(iii), the Taylor remainder theorem and the H\"{o}lder
inequality, we have
\begin{align*}
\lf|\partial^\az\lf(\mathscr{F}D_{A}^{i_0}a\r)(x)\r|
&=\lf|\int_{B_0}(-2\pi\imath\xi)^\az
\lf(D_{A}^{i_0}a\r)(\xi)e^{-2\pi\imath \xi\cdot x}\,d\xi\r|\\
&=\lf|\int_{B_0}(-2\pi\imath\xi)^\az
\lf(D_{A}^{i_0}a\r)(\xi)\lf[e^{-2\pi\imath \xi\cdot x}-P(\xi)\r]\,d\xi\r|\noz\\
&\ls\int_{B_0}|\xi^\az|
\lf|a\lf(A^{i_0}\xi\r)\r||x|^{s-|\az|+1}|\xi|^{s-|\az|+1}\,d\xi\noz\\
&\ls|x|^{s-|\az|+1}\int_{B_0}
\lf|a\lf(A^{i_0}\xi\r)\r||\xi|^{s+1}\,d\xi\noz\\
&\ls|x|^{s-|\az|+1}b^{-i_0}\int_{B_{i_0}}|a(\xi)|\,d\xi\noz\\
&\ls|x|^{s-|\az|+1}b^{-i_0/r}\|a\|_{L^r(\rn)}.\noz
\end{align*}

The other estimate can be directly deduced from the H\"{o}lder inequality, namely,
\begin{align*}
\lf|\partial^\az\lf(\mathscr{F}D_{A}^{i_0}a\r)(x)\r|
&=\lf|\int_{B_0}(-2\pi\imath\xi)^\az
\lf(D_{A}^{i_0}a\r)(\xi)e^{-2\pi\imath \xi\cdot x}\,d\xi\r|\\
&\ls\int_{B_0}
|\xi|^{|\az|}\lf|a\lf(A^{i_0}\xi\r)\r|\,d\xi
\ls b^{-i_0}\int_{B_{i_0}}|a(\xi)|\,d\xi
\ls b^{-i_0/r}\|a\|_{L^r(\rn)}.
\end{align*}
This finishes the proof of Lemma \ref{3l1}.
\end{proof}

Applying Lemma \ref{3l1}, we obtain a uniform estimate for $(\vec{p},r,s)$-atoms as follows, which plays a key role in the proof of Theorem \ref{3t1}.

\begin{lemma}\label{3l2}
Let $\vec{p}\in{(0,1]}^n$ and let $r$ and $s$ be as
in Definition \ref{3d3}. Then there exists a positive constant $C$ such that,
for any $(\vec{p},r,s)$-atom $a$ and $x\in\rn$,
\begin{align}\label{3e5}
\lf|\widehat{a}(x)\r|\leq C\max\lf\{\lf[\rho_*(x)\r]^{\frac1{p_-}-1},\,
\lf[\rho_*(x)\r]^{\frac1{p_+}-1}\r\},
\end{align}
where $\rho_*$ is as in Theorem \ref{3t1}.
\end{lemma}

The following inequalities will be used to prove Lemma \ref{3l2}, which are just \cite[p.\,11, Lemma 3.2]{bow03}.

\begin{lemma}\label{3l3}
Let $A$ be a given dilation. There exists a positive constant $C$ such that, for any $x\in\rn$,
\begin{equation*}
\frac{1}{C} [\rho(x)]^{\ln \lambda_-/\ln b} \leq \left|x\right|
\leq C[\rho(x)]^{\ln \lambda_+/\ln b}\qquad { when}\ \rho(x)\in(1,\fz),
\end{equation*}
and
\begin{equation*}
\frac{1}{C} [\rho(x)]^{\ln \lambda_+/\ln b} \leq \left|x\right|
\leq C[\rho(x)]^{\ln \lambda_-/\ln b}\qquad { when}\ \rho(x)\in[0,1],
\end{equation*}
where $\lambda_-$ and $\lambda_+$ are as in Definition \ref{2d1}.
\end{lemma}

We now give the proof of Lemma \ref{3l2}.

\begin{proof}[Proof of Lemma \ref{3l2}]
Let $a$ be a $(\vec{p},r,s)$-atom supported in $x_0+B_{i_0}$ with some $x_0\in\rn$
and $i_0\in\zz$. Without loss of generality, we may assume $x_0=\mathbf{0}$.
By Lemma \ref{3l1} with $\az=(\overbrace{0,\ldots,0}^{n\ \rm times})$, we find that,
for any $x\in\rn$,
\begin{align}\label{3e3}
\lf|\widehat{a}(x)\r|
&=\lf|b^{i_0}\lf(D_{A^*}^{i_0}\mathscr{F}D_{A}^{i_0}a\r)(x)\r|
=\lf|b^{i_0}\lf(\mathscr{F}D_{A}^{i_0}a\r)\lf((A^*)^{i_0}x\r)\r|\\
&\ls b^{i_0}b^{-i_0/r}\|a\|_{L^r(\rn)}\min\lf\{1,\lf|(A^*)^{i_0}x\r|^{s+1}\r\}\noz\\
&\ls b^{i_0}\lf\|\mathbf{1}_{B_{i_0}}\r\|_{\lv}^{-1}
\min\lf\{1,\lf|(A^*)^{i_0}x\r|^{s+1}\r\}\noz.
\end{align}

Next, we show that
\begin{align}\label{3e4}
\lf\|\mathbf{1}_{B_{i_0}}\r\|_{\lv}^{-1}\ls \max\lf\{b^{-\frac{i_0}{p_-}},\,b^{-\frac{i_0}{p_+}}\r\}.
\end{align}
Indeed, there exists a $K\in\zz$ large enough such that, if $i_0\in(K,\fz)\cap\zz$, then
\begin{align*}
\lf\|\mathbf{1}_{B_{i_0}}\r\|_{\lv}
&=\lf( \int_\rr\cdots \lf(\int_\rr|\mathbf{1}_{B_{i_0}}|^{p_1} d x_1\r)^{\frac{p_2}{p_1}}\cdots dx_n \r)^{\frac{1}{p_n}}\\
&\ge\lf(\int_\rr \cdots\lf(\int_\rr|\mathbf{1}_{B_{i_0}}|^{p_+} d x_1\r)^{\frac{p_+}{p_+}}\cdots dx_n \r)^{\frac{1}{p_+}}
=b^{\frac{i_0}{p_+}}.
\end{align*}
On the other hand, if $i_0\in(-\fz,K]$, by \cite[Lemma 6.8]{hlyy20}, we conclude that, for any $\varepsilon\in(0,1)$,
$$\frac{\|\mathbf{1}_{B_{K}}\|_{\lv}}{\|\mathbf{1}_{B_{i_0}}\|_{\lv}}
\ls b^{(K-i_0)\frac{1+\varepsilon}{p_-}}.$$
Letting $\varepsilon\to0$, we have
$$\lf\|\mathbf{1}_{B_{i_0}}\r\|_{\lv}^{-1} \ls \frac{b^\frac{K}{p_-}}{\|\mathbf{1}_{B_{K}}\|_{\lv}}b^{-\frac{i_0}{p_-}}.$$
Thus, \eqref{3e4} holds true. From this and \eqref{3e3}, it follows that, for any $x\in\rn$,

\begin{align}\label{3e6}
\lf|\widehat{a}(x)\r|
\ls b^{i_0}\max\lf\{b^{-\frac{i_0}{p_-}},\,b^{-\frac{i_0}{p_+}}\r\}
\min\lf\{1,\lf|(A^*)^{i_0}x\r|^{s+1}\r\}.
\end{align}
We next prove \eqref{3e5} by considering two cases:
$\rho_*(x)\leq b^{-i_0}$ and $\rho_*(x)>b^{-i_0}$.

\emph{Case 1).} $\rho_{*}(x)\leq b^{-i_0}$. In this case, note that
$\rho_*((A^*)^{i_0}x)\leq 1$. From \eqref{3e6}, Lemma \ref{3l3} and the fact that
$$1-\frac1{p_+}+(s+1)\frac{\ln\lambda_-}{\ln b}
\geq1-\frac1{p_-}+(s+1)\frac{\ln\lambda_-}{\ln b}>0,$$
we deduce that, for any $x\in\rn$ satisfying $\rho_*(x)\leq b^{-i_0}$,
\begin{align}\label{3e7}
\lf|\widehat{a}(x)\r|
&\ls b^{i_0}\max\lf\{b^{-\frac{i_0}{p_-}},\,b^{-\frac{i_0}{p_+}}\r\}
\lf[\rho_*\lf((A^*)^{i_0}x\r)\r]^{(s+1)\frac{\ln\lambda_-}{\ln b}}\\
&\sim \max\lf\{b^{i_0[1-\frac1{p_-}+(s+1)\frac{\ln\lambda_-}{\ln b}]},\,
b^{i_0[1-\frac1{p_+}+(s+1)\frac{\ln\lambda_-}{\ln b}]}\r\}
\lf[\rho_*(x)\r]^{(s+1)\frac{\ln\lambda_-}{\ln b}}\noz\\
&\ls\max\lf\{\lf[\rho_*(x)\r]^{\frac1{p_-}-1},\,
\lf[\rho_*(x)\r]^{\frac1{p_+}-1}\r\}.\noz
\end{align}
 This shows \eqref{3e5} for Case 1).

\emph{Case 2).} $\rho_*(x)>b^{-i_0}$. In this case, note that $\rho_*((A^*)^{i_0}x)>1$.
Using \eqref{3e6}, Lemma \ref{3l3} again and the fact that
$$\frac1{p_-}-1\geq\frac1{p_+}-1\geq0,$$
it is easy to see that, for any $x\in\rn$ satisfying $\rho_*(x)>b^{-i_0}$,
\begin{align*}
\lf|\widehat{a}(x)\r|
&\ls b^{i_0}\max\lf\{b^{-\frac{i_0}{p_-}},\,b^{-\frac{i_0}{p_+}}\r\}
\sim\max\lf\{b^{(1-\frac1{p_-})i_0},\,b^{(1-\frac1{p_+})i_0}\r\}\\
&\ls\max\lf\{\lf[\rho_*(x)\r]^{\frac1{p_-}-1},\,
\lf[\rho_*(x)\r]^{\frac1{p_+}-1}\r\},
\end{align*}
which completes the proof of \eqref{3e5} and hence of Lemma \ref{3l2}.
\end{proof}

\begin{lemma}\label{3l5}
Let $\vec{p}\in(0,1]^n$. Then, for any $\{\lz_i\}_{i\in\nn}\subset\mathbb{C}$ and $\{B^{(i)}\}_{i\in\nn}\subset\mathfrak{B}$,
$$\sum_{i\in\nn}|\lz_i|\le \lf\|\lf\{\sum_{i\in\nn}
\lf[\frac{|\lz_i|\mathbf{1}_{B^{(i)}}}{\|\mathbf{1}_{B^{(i)}}\|_{\lv}}\r]^
{\underline{p}}\r\}^{1/{\underline{p}}}\r\|_{\lv},$$
where $\underline{p}$ is as in \eqref{2e3}.
\end{lemma}

\begin{proof}
Observe that, for any $\{\lz_i\}_{i\in\nn}\subset\mathbb{C}$ and $\gamma\in(0,1]$,
\begin{align}\label{3e8}
\lf(\sum_{i=1}^{\fz}|\lz_i|\r)^{\gamma}\le \sum_{i=1}^{\fz}|\lz_i|^{\gamma}.
\end{align}
By this and the inverse Minkovski inequality, we know that
\begin{align*}
\lf\|\lf\{\sum_{i=1}^{\fz}
\lf[\frac{|\lz_i|\mathbf{1}_{B^{(i)}}}{\|\mathbf{1}_{B^{(i)}}\|_{\lv}}\r]^
{\underline{p}}\r\}^{1/{\underline{p}}}\r\|_{\lv}
&\ge\lf\|\sum_{i=1}^{\fz}
\frac{|\lz_i|\mathbf{1}_{B^{(i)}}}{\|\mathbf{1}_{B^{(i)}}\|_{\lv}}\r\|_{\lv}\\
&\ge\lf\|\sum_{i=1}^{N}
\frac{|\lz_i|\mathbf{1}_{B^{(i)}}}{\|\mathbf{1}_{B^{(i)}}\|_{\lv}}\r\|_{\lv}
\ge\sum_{i=1}^{N}|\lz_i|.
\end{align*}
Letting $N\to\fz$, we obtain the desired inequality as in Lemma \ref{3l5}.
\end{proof}

To show Theorem \ref{3t1}, we also need the following atomic characterizations
of $\vh$, which is just \cite[Theorem 4.7]{hlyy20}.

\begin{lemma}\label{3l4}
Let $\vec{p}\in{(0,\fz)}^n$, $r\in(\max\{p_+,1\},\fz]$ with $p_+$
as in \eqref{2e3}, $s$ be as in \eqref{3e2} and
$$N\in\mathbb{N}\cap\lf[\lf\lfloor\lf(\frac1{{{\rm min}\{p_-,1\}}}-1\r)
\frac{\ln b}{\ln\lambda_-}\r\rfloor+2,\fz\r)$$
with $p_-$ as in \eqref{2e3}.
Then $\vh=\vah$ with equivalent quasi-norms.
\end{lemma}

We now prove Theorem \ref{3t1}.

\begin{proof}[Proof of Theorem \ref{3t1}]
Let $\vec{p}\in{(0,1]}^n$, $r\in(\max\{p_+,1\},\fz]$, $s$ be as in \eqref{3e2}
and $f\in\vh$. Without loss of generality, we may assume that $\|f\|_{\vh}>0$. Then, by Lemma \ref{3l4} and Definition \ref{3d3}(\uppercase\expandafter{\romannumeral2}), we find that
there exist a sequence
$\{\lz_i\}_{i\in\nn}\subset\mathbb{C}$ and a sequence of $(\vec{p},r,s)$-atoms $\{a_i\}_{i\in\nn}$, supported respectively in
$\{B^{(i)}\}_{i\in\nn}\subset\B$, such that
\begin{align}\label{3e9}
f=\sum_{i\in\nn}\lz_ia_i
\quad\mathrm{in}\quad\cs'(\rn),
\end{align}
and
\begin{align}\label{3e10}
\|f\|_{\vh}\sim \lf\|\lf\{\sum_{i\in\nn}
\lf[\frac{|\lz_i|\mathbf{1}_{B^{(i)}}}{\|\mathbf{1}_{B^{(i)}}\|_{\lv}}\r]^
{\underline{p}}\r\}^{1/\underline{p}}\r\|_{\lv}.
\end{align}
Taking the Fourier transform on the both sides of \eqref{3e9}, we have
\begin{align}\label{3e11}
\widehat{f}=\sum_{i\in\nn}\lz_i\widehat{a_i}
\quad\mathrm{in}\quad\cs'(\rn).
\end{align}
Note that a function $f\in L^1(\rn)$ implies that
$\widehat{f}$ is well defined in $\rn$, so does $\widehat{a_i}$ for any $i\in\nn$. From
Lemmas \ref{3l2} and \ref{3l5}, and \eqref{3e10}, it follows that, for any $x\in\rn$,
\begin{align}\label{3e12}
\sum_{i\in\nn}|\lz_i||\widehat{a_i}(x)|
&\ls\sum_{i\in\nn}|\lz_i|
\max\lf\{\lf[\rho_*(x)\r]^{\frac1{p_-}-1},\,\lf[\rho_*(x)\r]^{\frac1{p_+}-1}\r\}\\
&\ls\|f\|_{\vh}
\max\lf\{\lf[\rho_*(x)\r]^{\frac1{p_-}-1},\,\lf[\rho_*(x)\r]^{\frac1{p_+}-1}\r\}
<\fz.\noz
\end{align}
Therefore, for any $x\in\rn$, the function
\begin{align}\label{3e13}
F(x):=\sum_{i\in\nn}\lz_i\widehat{a_i}(x)
\end{align}
is well defined pointwisely and
\begin{align*}
|F(x)|\ls\|f\|_{\vh}
\max\lf\{\lf[\rho_*(x)\r]^{\frac1{p_-}-1},\,\lf[\rho_*(x)\r]^{\frac1{p_+}-1}\r\}.
\end{align*}

We next show the continuity of the function $F$ on $\rn$. If we can prove that $F$ is continuous on any compact subset of $\rn$, then
the continuity on $\rn$ is obvious. Indeed, for any compact subset $E$, there exists a positive constant $K$, depending only on $A$ and
$E$, such that $\rho_*(x)\le K$ holds for every $x\in E$. By this and \eqref{3e12}, we conclude that, for any $x\in E$,

\begin{align*}
\sum_{i\in\nn}|\lz_i||\widehat{a_i}(x)|
&\ls\max\lf\{K^{\frac1{p_-}-1},\,K^{\frac1{p_+}-1}\r\}\|f\|_{\vh}
<\fz.
\end{align*}
Thus, the summation $\sum_{i\in\nn}\lz_i\widehat{a_i}(\cdot)$ converges uniformly
on $E$. This, together with the fact that, for any $i\in\nn$, $\widehat{a_i}(x)$ is continuous,
implies that $F$ is also continuous on any compact subset $E$ and hence on $\rn$.

Finally, to complete the proof of Theorem \ref{3t1}, by \eqref{3e11} and \eqref{3e13},
we only need to show that
\begin{align}\label{3e14}
F=\sum_{i\in\nn}\lz_i\widehat{a_i}
\quad\mathrm{in}\quad\cs'(\rn).
\end{align}
For this purpose, from Lemma \ref{3l2} and the definition of Schwartz functions, we deduce that,
for any $\varphi\in\cs(\rn)$ and $i\in\nn$,
\begin{align*}
&\lf|\int_{\rn}\widehat{a_i}(x)\varphi(x)\,dx\r|\\
&\quad\leq\sum_{k=1}^\fz\int_{(A^*)^{k+1}B^*_0\setminus(A^*)^{k}B^*_0}
\max\lf\{\lf[\rho_*(x)\r]^{\frac1{p_-}-1},\,\lf[\rho_*(x)\r]^{\frac1{p_+}-1}\r\}
|\varphi(x)|\,dx+\|\varphi\|_{L^1(\rn)}\\
&\quad\ls\sum_{k=1}^\fz b^kb^{k(\frac1{p_-}-1)}b^{-k(\lceil\frac1{p_-}-1\rceil+2)}
+\|\varphi\|_{L^1(\rn)}\\
&\quad\sim\sum_{k=1}^\fz b^{-k}+\|\varphi\|_{L^1(\rn)},
\end{align*}
where $B^*_0$ is the unit dilated ball with respect to $A^*$.
This implies that there exists a positive constant $C$ such that $\lf|\int_{\rn}\widehat{a_i}(x)\varphi(x)\,dx\r|\le C$ holds true
uniformly for any $i\in\zz$. Combining this, Lemma \ref{3l5} and \eqref{3e10}, we have
\begin{align*}
\lim_{I\to\fz}\sum_{i=I+1}^\fz|\lz_i|\lf|\int_{\rn}\widehat{a_i}(x)\varphi(x)\,dx\r|
\ls \lim_{I\to\fz}\sum_{i=I+1}^\fz|\lz_i|=0.
\end{align*}
Therefore, for any $\varphi\in\cs(\rn)$,
$$\langle F,\varphi\rangle
=\lim_{I\to\fz}\lf\langle\sum_{i=1}^I\lz_i\widehat{a_i},\,\varphi\r\rangle.$$
This finishes the proof of \eqref{3e14} and hence of Theorem \ref{3t1}.
\end{proof}

\begin{remark}\label{3r1}
\begin{enumerate}
\item[(i)]
When $\vec{p}=(p,\ldots,p)\in(0,1]^n$, the Hardy space $\vh$ in Theorem \ref{3t1} coincides with
the anisotropic Hardy space $H^p_A(\rn)$ from \cite{bow03}, and the inequality \eqref{3e1}
becomes
\begin{align*}
|F(x)|\le C\|f\|_{H^{p}_A(\rn)}[\rho_\ast(x)]^{\frac1{p}-1}
\end{align*}
with $C$ as in \eqref{3e1}. In this case, Theorem \ref{3t1} is just \cite[Theorem 1]{bw13}.
\item[(ii)]
Let $f\in \vh\cap L^1({\mathbb{R}^n})$. By the inequality \eqref{3e1}
with $x=\mathbf{0}$, we obtain $F=\widehat{f}$ and $\widehat{f}(\mathbf{0})=0$.
Thus, the function $f\in \vh\cap L^1({\mathbb{R}^n})$
has a vanishing moment, which illustrates the necessity of the vanishing moment of
atoms in some sense.
\item[(iii)]
Very recently, in \cite[Theorem 2.4]{hcy21}, Huang et al. obtained a result similar
to Theorem \ref{3t1} in the setting of the anisotropic mixed-norm
Hardy space ${H_{\vec{a}}^{\vec{p}}(\rn)}$, where
$$\vec{a}:=(a_1,\ldots,a_n)\in [1,\fz)^n\quad
{\rm and}\quad \vec{p}:=(p_1,\ldots, p_n)\in (0,1]^n.$$
We should point out that if
$$A:=\begin{pmatrix}
         2^{a_1} & 0 & \cdots & 0 \\
         0 & 2^{a_2} & \cdots & 0 \\
         \vdots & \vdots &\ddots & \vdots \\
         0 & 0 & \cdots & 2^{a_n} \\
       \end{pmatrix},
$$
then $\vh$=${H_{\vec{a}}^{\vec{p}}(\rn)}$ with equivalent quasi-norms.
In this sense, Theorem \ref{3t1} covers \cite[Theorem 2.4]{hcy21} as a special case.
\end{enumerate}
\end{remark}

\section{Applications\label{s4}}

As applications of Theorem \ref{3t1}, we first prove the function $F$ given in Theorem \ref{3t1} has a higher order convergence at the origin. Then we extend the Hardy--Littlewood inequality to the setting of anisotropic mixed-norm
Hardy spaces.

We embark on the proof of the first desired result.

\begin{theorem}\label{4t1}
Let $\vec{p}\in{(0,1]}^n$. Then, for any
$f\in\vh$, there exists a continuous function $F$ on $\rn$ such that
$\widehat{f}=F$ in $\cs'(\rn)$
and
\begin{align}\label{4e1}
\lim_{|x|\to0^+}\frac{F(x)}{[\rho_*(x)]^{\frac1{p_+}{-1}}}=0.
\end{align}
\end{theorem}

\begin{proof}
Let $\vec{p}\in(0,1]^n$, $r\in(\max\{p_+,1\},\fz]$, $s$ be as in \eqref{3e2}
and $f\in\vh$. Then, by Lemma \ref{3l4} and
Definition \ref{3d3}(\uppercase\expandafter{\romannumeral2}), we find that there exists a sequence
$\{\lz_i\}_{i\in\nn}\subset\mathbb{C}$ and a sequence of $(\vec{p},r,s)$-atoms,
$\{a_i\}_{i\in\nn}$, supported, respectively, in
$\{B^{(i)}\}_{i\in\nn}\subset\B$ such that
\begin{align*}
f=\sum_{i\in\nn}\lz_ia_i
\quad\mathrm{in}\quad\cs'(\rn),
\end{align*}
and
\begin{align}\label{4e2}
\|f\|_{\vh}\sim \lf\|\lf\{\sum_{i\in\nn}
\lf[\frac{|\lz_i|\mathbf{1}_{B^{(i)}}}{\|\mathbf{1}_{B^{(i)}}\|_{\lv}}\r]^
{\underline{p}}\r\}^{1/\underline{p}}\r\|_{\lv}.
\end{align}
Furthermore, from the proof of Theorem \ref{3t1}, it follows that the function
\begin{align}\label{4e3}
F(x)=\sum_{i\in\nn}\lz_i\widehat{a_i}(x), \quad \forall x\in\rn,
\end{align}
is continuous and satisfies that $\widehat{f}=F$ in $\cs'(\rn)$.

Thus, to show Theorem \ref{4t1}, we only need to prove \eqref{4e1} holds true for the function $F$
as in \eqref{4e3}. To do this, observe that, for any $(\vec{p},r,s)$-atom
$a$ supported in $x_0+B_{k_0}$ with some $x_0\in\rn$ and $k_0\in\zz$, when $\rho_*(x)\leq b^{-k_0}$,
\eqref{3e7} holds true. This, together  with the fact that
$$1-\frac1{p_+}+(s+1)\frac{\ln\lambda_-}{\ln b}>0,$$
implies that
\begin{align}\label{4e5}
\lim_{|x|\to0^+}\frac{|\widehat{a}(x)|}{[\rho_*(x)]^{\frac1{p_+}{-1}}}=0.
\end{align}
For any $x\in\rn$, we get the following inequality by \eqref{4e3}:
\begin{align}\label{4e6}
\frac{|F(x)|}{[\rho_*(x)]^{\frac1{p_+}{-1}}}
\leq\sum_{i\in\nn}|\lz_i|\frac{|\widehat{a_i}(x)|}{[\rho_*(x)]^{\frac1{p_+}{-1}}}.
\end{align}
Moreover, by \eqref{3e5} and the fact $\sum_{i\in\nn}|\lz_i|<\fz$,
we know that the dominated convergence theorem can be applied to the right side of \eqref{4e6}.
Combining this and \eqref{4e5}, we deduce that

\begin{align*}
\lim_{|x|\to0^+}\frac{F(x)}{[\rho_*(x)]^{\frac1{p_+}{-1}}}=0,
\end{align*}
which completes the proof of Theorem \ref{4t1}.
\end{proof}

\begin{remark}
\begin{enumerate}
\item[(i)]
Similarly to Remark \ref{3r1}(i),
if $\vec{p}=(p,\ldots,p)\in(0,1]^n$, then the Hardy space $\vh$ in Theorem \ref{4t1} coincides with
the anisotropic Hardy space $H^p_A(\rn)$ from \cite{bow03}. In this case, Theorem \ref{4t1} is just \cite[Corollary 6]{bw13}.
\item[(ii)]
By Theorem \ref{4t1} and Lemma \ref{3l3}, we have
\begin{align}\label{4e7}
\lim_{|x|\to 0^+}\frac{F(x)}
{|x|^{\frac{\ln b}{\ln\lambda_+}(\frac 1{p_+}-1)}}=0.
\end{align}
Observe that, when $\vec{p}=(p,\ldots,p)\in(0,1]^n$ and
$A=d\,{\rm I}_{n\times n}$ for some $d\in\rr$ with $|d|\in(1,\fz)$,
here and thereafter, ${\rm I}_{n\times n}$ denotes the \emph{unit matrix} of order $n$,
the Hardy space $\vh$ comes back to the classical Hardy space $H^p({\mathbb{R}^n})$ of
Fefferman and Stein \cite{fs72}. In this case, $\frac{\ln b}{\ln\lambda_+}=n$ and $p_+=p$,
and hence \eqref{4e7} is just the well-known result on $H^p({\mathbb{R}^n})$
(see \cite[p.\,128]{ste93}).
\end{enumerate}
\end{remark}

As another application of Theorem \ref{3t1}, we extend the
Hardy--Littlewood inequality to the setting of anisotropic
mixed norm Hardy spaces in the following theorem.

\begin{theorem}\label{4t2}
Let $\vec{p}\in(0,1]^n$. Then, for any
$f\in\vh$, there exists a continuous function $F$ on $\rn$ such that
$\widehat{f}=F$ in $\cs'(\rn)$
and
\begin{align}\label{4e8}
\lf(\int_\rn|F(x)|^{p_+}\min\lf\{\lf[\rho_*(x)\r]^{p_+-\frac{p_+}{p_-}-1},\,
\lf[\rho_*(x)\r]^{p_+-2}\r\}\,dx\r)^{\frac1{p_+}}
\leq C\|f\|_{\vh},
\end{align}
where $C$ is a positive constant depending only on $A$ and $\vec{p}$.
\end{theorem}

\begin{proof}
Let $\vec{p}\in(0,1]^n$ and $f\in\vh$. Then, by Lemma \ref{3l4}
and Definition \ref{3d3}(\uppercase\expandafter{\romannumeral2}), we find that
there exist a sequence
$\{\lz_i\}_{i\in\nn}\subset\mathbb{C}$ and a sequence of $(\vec{p},2,s)$-atoms
$\{a_i\}_{i\in\nn}$, supported respectively, in
$\{B^{(i)}\}_{i\in\nn}\subset\B$ such that
\begin{align*}
f=\sum_{i\in\nn}\lz_ia_i
\quad\mathrm{in}\quad\cs'(\rn),
\end{align*}
and
\begin{align}\label{4e9}
\lf\|\lf\{\sum_{i\in\nn}
\lf[\frac{|\lz_i|\mathbf{1}_{B^{(i)}}}{\|\mathbf{1}_{B^{(i)}}\|_{\lv}}\r]^
{\underline{p}}\r\}^{1/\underline{p}}\r\|_{\lv}\ls\|f\|_{\vh}<\fz.
\end{align}
To prove Theorem \ref{4t2}, it suffices to show that \eqref{4e8}
holds true for the function $F$ as in \eqref{4e3}. For this purpose,
by the fact that $\underline{p}\le p_+\le1$, the inverse Minkovski inequality and \eqref{4e9}, we have
\begin{align}\label{4e12}
\left(\sum_{i\in{\mathbb N}}|\lambda_i|^{p_+}\right)^{1/p_+}
&=\left(\sum_{i\in{\mathbb N}}\left\|\frac{|\lambda_i|{\mathbf 1}_{B^{(i)}}}
{\|{\mathbf 1}_{B^{(i)}}\|_{L^{\vec{p}}(\mathbb{R}^n)}}
\right\|_{\lv}^{p_+}\right)^{1/p_+}\\
&=\left(\sum_{i\in{\mathbb N}}
\left\|\frac{|\lambda_i|^{p_+}{\mathbf 1}_{B^{(i)}}}{\|{\mathbf 1}_{B^{(i)}}\|_{\lv}^{p_+}}
\right\|_{L^{\vec{p}/p_+}({\mathbb{R}^n})}\right)^{1/p_+}\nonumber\\
&\le \left\|\sum_{i\in{\mathbb N}}
\left[\frac{|\lambda_i|{\mathbf 1}_{B^{(i)}}}{\|{\mathbf 1}_{B^{(i)}}\|_{\lv}}
\right]^{p_+}\right\|_{L^{\vec{p}/p_+}({\mathbb{R}^n})}^{1/p_+}\nonumber\\
&=\lf\|\lf\{\sum_{i\in{\mathbb N}}\lf[\frac{|\lambda_i|{\mathbf 1}_{B^{(i)}}}
{\|{\mathbf 1}_{B^{(i)}}\|_{\lv}}
\right]^{p_+}\right\}^{1/p_+}\right\|_{L^{\vec{p}}(\mathbb{R}^n)}\nonumber\\
&\le\left\|\left\{\sum_{i\in{\mathbb N}}
\left[\frac{|\lambda_i|{\mathbf 1}_{B^{(i)}}}{\|{\mathbf 1}_{B^{(i)}}\|_
{\lv}}\right]^{\underline{p}}\right\}^{1/\underline{p}}
\right\|_{\lv}\nonumber\\
&\ls\|f\|_{\vh}.\nonumber
\end{align}

On another hand, from \eqref{4e3}, the fact that $p_+\in(0,1]$, \eqref{3e8}
and the Fatou lemma, it follows that
\begin{align}\label{4e13}
&\int_{{\mathbb{R}^n}}|F(x)|^{p_+}
\min\left\{\left[\rho_{*}(x)\right]^{p_+-\frac {p_+}{p_-}-1},\,
\left[\rho_{*}(x)\right]^{p_+-2}\right\}\, dx\\
&\quad\le \sum_{i\in{\mathbb N}}|\lambda_i|^{p_+}\int_{{\mathbb{R}^n}}\left[\lf|\widehat{a_i}(x)\r|
\min\left\{\left[\rho_{*}(x)\right]^{1-\frac 1{p_-}-\frac 1{p_+}},\,
\left[\rho_{*}(x)\right]^{1-\frac 2{p_+}}\right\}\right]^{p_+}\,dx.\noz
\end{align}
Next, we devote to proving the following uniform estimate for all $(\vec{p},2,s)$-atoms, namely,
\begin{align}\label{4e14}
\left(\int_{{\mathbb{R}^n}}\left[\lf|\widehat{a}(x)\r|
\min\left\{\left[\rho_{*}(x)\right]^{1-\frac 1{p_-}-\frac 1{p_+}},\,
\left[\rho_{*}(x)\right]^{1-\frac 2{p_+}}
\right\}\right]^{p_+}\,dx\right)^{1/p_+}\le M,
\end{align}
where $M$ is a positive constant independent of $a$. Assume that \eqref{4e14} holds true for the moment. Combining this, \eqref{4e12} and \eqref{4e13}, we conclude that
\begin{align*}
&\lf(\int_{{\mathbb{R}^n}}|F(x)|^{p_+}
\min\left\{\left[\rho_{*}(x)\right]^{p_+-\frac {p_+}{p_-}-1},\,
\left[\rho_{*}(x)\right]^{p_+-2}\right\}\,dx\r)^{1/{p_+}}\\
&\quad\le M\left(\sum_{i\in{\mathbb N}}|\lambda_i|^{p_+}\right)^{1/{p_+}}
\ls\|f\|_{\vh}.
\end{align*}
This is the desired conclusion \eqref{4e8}.

Thus, the rest of the whole proof is to show the assertion
\eqref{4e14}. Indeed, for any $(\vec{p},2,s)$-atom $a$ supported in
a dilated ball $x_0+B_{i_0}$ with some $x_0\in\rn$ and $i_0\in\zz$,
it is easy to see that
\begin{align*}
&\left(\int_{{\mathbb{R}^n}}\left[\lf|\widehat{a}(x)\r|
\min\left\{\left[\rho_{*}(x)\right]^{1-\frac 1{p_-}-\frac 1{p_+}},\,
\left[\rho_{*}(x)\right]^{1-\frac 2{p_+}}\right\}\right]^{p_+}\,dx\right)^{1/{p_+}}\\
&\quad\ls\left(\int_{(A^*)^{-i_0+1}B_0^*}\left[\lf|\widehat{a}(x)\r|
\min\left\{\left[\rho_{*}(x)\right]^{1-\frac 1{p_-}-\frac 1{p_+}},\,
\left[\rho_{*}(x)\right]^{1-\frac 2{p_+}}\right\}\right]^{p_+}\,dx\right)^{1/{p_+}}\\
&\qquad
+\left(\int_{((A^*)^{-i_0+1}B_0^*)^{\complement}}\left[\lf|\widehat{a}(x)\r|
\min\left\{\left[\rho_{*}(x)\right]^{1-\frac 1{p_-}-\frac 1{p_+}},\,
\left[\rho_{*}(x)\right]^{1-\frac 2{p_+}}\right\}\right]^{p_+}\,dx\right)^{1/{p_+}}\\
&\quad=:{\rm I}_1+{\rm I}_2,
\end{align*}
where $B^*_0$ is the unit dilated ball with respect to $A^*$.

Let $\theta$ be a fixed positive constant such that
$$1-\frac1{p_+}+(s+1)\frac{\ln\lambda_-}{\ln b}-\theta
\geq1-\frac1{p_-}+(s+1)\frac{\ln\lambda_-}{\ln b}-\theta>0.$$
Then, to deal with ${\rm I}_1$, by \eqref{3e7}, we know that
\begin{align*}
{\rm I}_1
&\ls b^{i_0[1+(s+1)\frac{\ln\lambda_-}{\ln b}]}
\max\lf\{b^{-\frac{i_0}{p_-}},\,b^{-\frac{i_0}{p_+}}\r\}\\
&\hs\times\left(\int_{(A^*)^{-i_0+1}B_0^*}\left[
\min\left\{\left[\rho_{*}(x)\right]^{1-\frac 1{p_-}-\frac 1{p_+}+(s+1)\frac{\ln\lambda_-}{\ln b}},\,
\left[\rho_{*}(x)\right]^{1-\frac 2{p_+}+(s+1)\frac{\ln\lambda_-}{\ln b}}\right\}\right]^{p_+}\,dx\right)^{1/{p_+}}\\
&\ls b^{i_0[1+(s+1)\frac{\ln\lambda_-}{\ln b}]}
\max\lf\{b^{-\frac{i_0}{p_-}},\,b^{-\frac{i_0}{p_+}}\r\}\\
&\hs\times\min\lf\{b^{-i_0[1-\frac1{p_-}+(s+1)
\frac{\ln\lambda_-}{\ln b}-\theta]},\,b^{-i_0[1-\frac1{p_+}+(s+1)
\frac{\ln\lambda_-}{\ln b}-\theta]}\r\}\left(\int_{(A^*)^{-i_0+1}B_0^*}
\left[\rho_{*}(x)\right]^{\theta p_+-1}\,dx\right)^{1/{p_+}}\\
&\sim b^{i_0\theta}\left[\sum_{k\in\zz\setminus\nn}b^{-i_0+k}(b-1)b^{(-i_0+k)(\theta p_+-1)}\r]^{1/{p_+}}\sim\lf(\frac{b-1}{1-b^{-\theta p_+}}\r)^{\frac1{p_+}}.
\end{align*}
As for the estimate of ${\rm I}_2$, by the H\"{o}lder inequality, the Plancherel theorem,
the fact that $0<p_-\le p_+\le1$ and the size condition of $a$, we obtain
\begin{align*}
{\rm I}_2
&\ls\left\{\int_{((A^*)^{-i_0+1}B_0^*)^{\complement}}
\lf|\widehat{a}(x)\r|^2\,dx\right\}^{\frac12}\\
&\hs\hs\times\left\{\int_{((A^*)^{-i_0+1}B_0^*)^{\complement}}
\left[\min\left\{\left[\rho_{*}(x)\right]^{1-\frac 1{p_-}-\frac 1{p_+}},\,
\left[\rho_{*}(x)\right]^{1-\frac 2{p_+}}\right\}\right]^{\frac{2p_+}{2-p_+}}\,dx\right\}^{\frac{2-p_+}{2p_+}}\\
&\ls\|a\|_{L^2({\mathbb{R}^n})}\left\{\sum_{k\in\nn} b^{-i_0+k}(b-1)
\left[\min\left\{b^{(-i_0+k)(1-\frac 1{p_-}-\frac 1{p_+})},\,
b^{(-i_0+k)(1-\frac 2{p_+})}\right\}\right]^{\frac{2p_+}{2-p_+}}\right\}^{\frac{2-p_+}{2p_+}}\\
&\ls\|a\|_{L^2({\mathbb{R}^n})}\left\{b^{-i_0}
\left[\min\left\{b^{-i_0(1-\frac 1{p_-}-\frac 1{p_+})},\,
b^{-i_0(1-\frac 2{p_+})}\right\}\right]^{\frac{2p_+}{2-p_+}}\right\}^{\frac{2-p_+}{2p_+}}\\
&\ls\max\left\{b^{i_0(\frac12-\frac 1{p_-})},\,
b^{i_0(\frac12-\frac 1{p_+})}\right\}
\min\left\{b^{-i_0(\frac12-\frac 1{p_-})},\,
b^{-i_0(\frac12-\frac 1{p_+})}\right\}\\
&\sim 1.
\end{align*}
This finishes the proof of \eqref{4e14} and hence of Theorem \ref{4t2}.
\end{proof}

\begin{remark}
Actually, when $\vec{p}=(p,\ldots,p)\in(0,1]^n$, the Hardy space $\vh$ in Theorem \ref{4t2} is just
the anisotropic Hardy space $H^p_A(\rn)$ from \cite{bow03} in the sense of equivalent quasi-norms. Thus, we point out that Theorem \ref{4t2} covers \cite[Corollary 8]{bw13}.
Moreover, if $A=d\,{\rm I}_{n\times n}$ for some $d\in\rr$ with $|d|\in(1,\fz)$,
then the anisotropic mixed-norm Hardy space $H^{\vec{p}}_A(\rn)$, with
$\vec{p}=(p,\ldots,p)\in(0,1]^n$, coincides with the classical Hardy space $H^p({\mathbb{R}^n})$ of
Fefferman and Stein \cite{fs72}. In this case, $\rho_*(x)\sim|x|^n$ for any $x\in\rn$, and hence \eqref{4e6} is just the classic Hardy--Littlewood inequality as in \eqref{1e2}.
\end{remark}

\bigskip

\noindent  Jun Liu (Corresponding author), Yaqian Lu and Mingdong Zhang

\medskip

\noindent  School of Mathematics,
China University of Mining and Technology,
Xuzhou 221116, Jiangsu, People's Republic of China

\smallskip

\noindent {\it E-mails}: \texttt{junliu@cumt.edu.cn} (J. Liu)

\hspace{1cm}\texttt{yaqianlu@cumt.edu.cn} (Y. Lu)

\hspace{1cm}\texttt{mdzhang@cumt.edu.cn} (M. Zhang)
\end{document}